\documentclass[12pt,reqno]{amsart} 
\usepackage{amsmath,amssymb,amsthm,mathrsfs}
\usepackage[english]{babel}
\usepackage[utf8]{inputenc}
\usepackage{xcolor}

\newtheorem{theorem}{Theorem}[section]
\newtheorem{proposition}[theorem]{Proposition}

\theoremstyle{definition}
\newtheorem{definition}[theorem]{Definition}

\usepackage{multirow} 
\usepackage{color}

\usepackage{graphicx}
\usepackage[left=3cm,right=3cm]{geometry}

\def\r{\mathbb R}
\def\h{\mathbb H}
\def\s{\mathbb S}

\begin{document}

\title{Solitons of the mean curvature flow in $\s^2\times\r$ }
\author{Rafael L\'opez, Marian Ioan Munteanu}
\address{ Departamento de Geometr\'{\i}a y Topolog\'{\i}a\\  Universidad de Granada. 18071 Granada, Spain}
\email{rcamino@ugr.es}
\address{University Al. I. Cuza' of Iasi, Faculty of Mathematics, Bd. Carol I, no. 11, 700506 Iasi, Romania}
\email{marian.ioan.munteanu@gmail.com}

\begin{abstract}
A    soliton of the mean curvature flow in the product space $\s^2\times\r$ is a surface whose mean curvature $H$ satisfies the equation $H=\langle N,X\rangle$, where $N$ is the unit normal of the surface and $X$ is a Killing vector field. In this paper we consider  the vector field   tangent to the fibers and the vector field associated to rotations about an axis of $\s^2$, respectively.  We give a classification of the solitons with respect to these vector fields   assuming that the surface is invariant  under a one-parameter  group of vertical translations or  rotations of $\s^2$.
 \end{abstract}


\maketitle

\section{Introduction}
 
Let   $\psi\colon\Sigma\to\r^3$  be an immersion of a surface $\Sigma$ in Euclidean space $\r^3$. A variation $\{ \psi_t\colon\Sigma\to \r^3:t\in  [0,T)\}$, $T>0$, $\psi_0=\psi$, evolves by the mean curvature flow (MCF in short) if $\displaystyle\frac{\partial\psi_t}{\partial t}=H(\psi_t)N(\psi_t)$, where  $H(\psi_t)$ is the mean curvature of $\psi_t$ and $N(\psi_t)$ is its unit normal.   The surface $\Sigma$ is called a soliton of MCF   if the evolution of $\Sigma$ under a one-parameter family of dilations or isometries remains constant. An important type of solitons are the translators whose shape  is invariant by  translations   along a direction $\vec{v}\in\r^3$. Translators are   characterized by the equation $H=\langle N,\vec{v}\rangle$, where $H$ and $N$ are the mean curvature and unit normal of $\Sigma$ respectively.  Translators play a special role in the theory of MCF because they are, after rescaling,  a type of singularities of the MCF  according to Huisken and Sinestrari  \cite{hs}.   
In the meantime, the development of the theory of solitons of the MCF in other ambient space has been developed. Without to be complete, we refer:  a general product space $M^2\times \r$ \cite{lima}; hyperbolic space \cite{bl1,bl2,li,mr};  the product $\h^2\times\r$ \cite{bue1,bue2,bl3,lipi}; the Sol space \cite{pi1}; the Heisenberg group \cite{pi2}; the special linear group \cite{lm1}. 

In this paper, we focus on solitons in the product space $\s^2\times\r$, where $\s^2$ is the unit sphere of $\r^3$. Looking in the equation  $H=\langle N,\vec{v}\rangle$, it is natural to replace $\vec{v}$ by a Killing vector of the space, which motivates the following definition. 

\begin{definition}\label{def1} Let $X\in\mathfrak{X}(\s^2\times\r)$ be a Killing vector field. A surface $\Sigma$ in $\s^2\times\r$ is said to be a $X$-soliton   if its mean curvature $H$ and unit normal vector $N$ satisfy
\begin{equation}\label{eq1}
H=\langle N,X\rangle.
\end{equation}
\end{definition}

In this paper, $H$ is the sum of the principal curvatures of the surface. The dimension of Killing vector fields in the space $\s^2\times\r$ is $4$. Taking coordinates $(x,y,z,t)$ in $\s^2\times\r$, a relevant Killing vector field is $V=\partial_t$.  Here $V$ is tangent to the fibers of the natural submersion $\s^2\times\r\to\s^2$. Other Killing vector fields come from the rotations of $\s^2$.   After renaming coordinates, consider the vector field   $R=-y\partial_x+x\partial_y$  about the $z$-axis of $\s^2$. Examples of solitons are the following.  
\begin{enumerate}
\item  Cylinders over geodesic of $\s^2$ are $V$-solitons. Indeed, let $\Sigma=C\times\r$ be a surface constructed as a cylinder over a curve $C$ of $\s^2$. Then the mean curvature of $\Sigma$ is $H=\kappa$, where $\kappa$ the curvature of $C$. Since the unit normal vector $N$ of $\Sigma$ is orthogonal to $\partial_t$, then $\langle N,V\rangle=0$. Thus  $\Sigma$ is a $V$-soliton if and only if   $\kappa=0$. Thus $C$ is a geodesic of $\s^2$. 
\item Slices $\s^2\times\{t_0\}$, $t_0\in\r$, are $R$-solitons. Notice that $H=0$ because a slice is totally geodesic.  Since $N=\partial_t$, then $\langle N,R\rangle=0$, proving that $H=\langle N,R\rangle$.
\end{enumerate}
 
 In this paper, we are interested in examples of $V$-solitons and $R$-solitons that are invariant by a one-parameter group of isometries of $\s^2\times\r$. Here we consider two types  of such surfaces. First, vertical surfaces which are invariant by vertical translations in the $t$-coordinate. Second, rotational surfaces, which are invariant by a group of rotations about an axis of $\s^2$. Under these geometric conditions on the surfaces, we give a full classifications of $V$-solitons (Sect. \ref{sec3}) and $R$-solitons (Sect. \ref{sec4}). 
 
 \section{Preliminaries}
 
 In this section, we compute each one of the terms of Eq. \eqref{eq1} for vertical and rotational surfaces. 
  The isometry group of $\s^2\times\r$ isomorphic to 
  $\mbox{Isom}(\s^2)\times\mbox{Isom}(\r)$.    The group $\mbox{Isom}(\s^2)$ is generated by   the identity, the antipodal map, rotations and reflections. The group $\mbox{Isom}(\r)$ contains the identity, translations, and reflections. 
 Therefore there are two  important one-parameter groups of isometries in $\s^2\times\r$:   vertical translations in the factor $\r$ and rotations in the factor $\s^2$. This leads two types of  invariant surfaces.  
 
  \begin{enumerate}
  \item {\it Vertical surfaces}. A vertical translation is a map of type   $T_\lambda\colon\s^2\times\r\to\s^2\times \r$ defined by $T_\lambda(p,t)=(p,t+\lambda)$, where $\lambda$ is fixed. This defines a one-parameter group of vertical translations  $\mathcal{T}=\{T_\lambda:\lambda\in\r\}$. 
    A    vertical surface  is a surface  $\Sigma$   invariant by the group $\mathcal{T}$, that is,  $T_\lambda(\Sigma)\subset\Sigma$ for all $\lambda\in\r$. The generating curve of $\Sigma$ is a curve $\alpha\colon I\subset\r\to\s^2$ in the unit sphere $\s^2$. Let us  write  this curve as 
\begin{equation}\label{a1}
\alpha(s)=(\cos u(s)\cos  v(s) ,\cos u(s) \sin  v(s),\sin u(s)),
\end{equation}
for some smooth functions $u=u(s)$ and $v=v(s)$.  Then a parametrization of $\Sigma$ is
\begin{equation}\label{p1}
\Psi(s,t)=(\cos u(s)\cos  v(s) ,\cos u(s) \sin  v(s),\sin u(s),t),\quad s\in I, t\in\r.
\end{equation}
In what follows, we parametrize the curve $\beta(s)=(u(s),v(s))$  to have 
$$ u'(s)=\cos u(s)\cos\theta(s),\quad  v'(s)= \sin\theta(s).$$

  \item {\it Rotational surfaces}. These surfaces are invariant by rotations of $\s^2$. To be precise, and after a choice of coordinates on $\s^2$, a rotation in $\s^2\times\r$ about the $z$-axis is a map $\mathcal{R}_\varphi\colon\s^2\times\r\to\s^2\times\r$, given by $$R_\varphi=\left(\begin{array}{llll} \cos \varphi&-\sin \varphi&0&0 \\ \sin \varphi&\cos \varphi&0&0\\ 0&0& 1&0\\ 0&0&0&1\end{array}\right).$$
The set  $\mathcal{R}=\{\mathcal{R}_\varphi:\varphi\in\r\}$ of all $\mathcal{R}_\varphi$, is a one-parameter group of rotations, that is ${\mathrm{SO}}(2)$.   A  rotational surface   is a surface $\Sigma$ invariant by the   group $\mathcal{R}$, that is, $\mathcal{R}_\varphi(\Sigma)\subset\Sigma$ for all $\varphi\in\r$.    The generating curve of $\Sigma$ is a curve $\alpha$ contained in the $xzt$-hyperplane which we suppose parametrized by 
\begin{equation}\label{a2}
\alpha(s)=(\cos u(s),0,\sin u(s), v(s)),\quad s\in I\subset\r,
\end{equation}
where $u=u(s)$ and $v=v8s)$ are smooth functions.  Then a parametrization of $\Sigma$ is 
\begin{equation}\label{p2}
\Psi(s,\varphi)=(\cos u(s)\cos \varphi,\cos u(s)\sin \varphi,\sin u(s), v(s)),\quad s\in I, \varphi\in \r.
\end{equation}
From now, suppose  that the   the curve $\beta(s)=( u(s), v(s))$ obtained from $\alpha$ in \eqref{a2} is parametrized by the Euclidean arc-length, that is, 
$$  u'(s)=\cos\theta(s),\quad  v'(s)=\sin\theta(s),$$
for some smooth function $\theta=\theta(s)$.   Notice that  $\theta'$ is the curvature of $\beta$ as planar curve of $\r^2$. 
  \end{enumerate}
 
 We now compute the mean curvature $H$ and the unit normal vector $N$ of vertical surfaces and rotational surfaces. 
 
 \begin{proposition}\label{pr1}
   Suppose that $\Sigma$ is a vertical   surface parametrized by \eqref{p1}.  Then the unit normal vector $N$ is 
 \begin{equation}\label{n1}
 N=(\cos \theta \sin v-\sin \theta \sin u \cos v,-\cos \theta \cos v-\sin \theta \sin u \sin v,\sin \theta \cos u,0),
 \end{equation}
 and the mean curvature $H$ is 
 \begin{equation}\label{h1}
H=\tan u\sin\theta-\frac{\theta '}{\cos u}.
\end{equation}   
\end{proposition}

\begin{proof}
Suppose that $\Sigma$ is parametrized by \eqref{p1}.  Then the tangent plane each point of $\Sigma$ is spanned by $\{\Psi_s,\Psi_t\}$, where 
\begin{equation}
\begin{split}
\Psi_s&=(-\cos u (\cos \theta \sin u \cos v+\sin \theta \sin v),\cos u (\sin \theta \cos v-\cos \theta \sin u \sin v),\cos \theta \cos ^2u,0),\\
\Psi_t&=(0,0,0,1).
\end{split}
\end{equation} 
A straightforward computation yields that the unit normal vector is \eqref{n1}. 

As usually, denote by $g_{ij}$ the coefficients of the first fundamental form of $\Psi$, where 
$$g_{11}=\langle\Psi_s,\Psi_s\rangle,\quad g_{12}=\langle\Psi_s,\Psi_t\rangle,\quad g_{22}=\langle\Psi_t,\Psi_t\rangle.$$
The formula of $H$ is 
$$H=\frac{g_{22}b_{11}-2g_{12}b_{12}+g_{11}b_{22}}{g_{11}g_{22}-{g_{12}}^2},$$
where $b_{ij}$ are the coefficients of the second fundamental form. Here 
$$b_{11}=-\langle N_s,\Psi_s\rangle,\quad   b_{12}=-\langle N_s,\Psi_t\rangle,\quad b_{22}=-\langle N_t,\Psi_t\rangle.$$ 
A computation of $g_{ij}$ gives  $g_{11}=  g_{22}=(\cos u)^2$ and $g_{12}=0$. In particular, $\cos u(s)\not=0$ for all $s\in I$. Then $g_{11}g_{22}-g_{12}^2=(\cos u)^4$. For the coefficients of the second fundamental, we have  $b_{12}=b_{22}=0$ and  
\begin{equation}
b_{11}=\cos u  (\sin\theta \sin -\theta ' ).
\end{equation}
Then the mean curvature $H$ is \eqref{h1}.
 \end{proof}

 \begin{proposition}\label{pr2}
   Suppose that $\Sigma$ is a rotational   surface parametrized by \eqref{p2}.  Then the unit normal vector $N$ is 
 \begin{equation}\label{n2}
N(s,\varphi)= ( \sin \theta \sin  u \cos \varphi, \sin \theta  \sin  u \sin \varphi ,-\sin \theta  \cos u,\cos \theta ),
\end{equation}
 and the mean curvature $H$ is 
\begin{equation}\label{h2}
H=\theta'-\sin\theta\tan u.
\end{equation}  
\end{proposition}

\begin{proof} 

 From \eqref{p1}, the basis $\{\Psi_s,\Psi_t\}$ at each tangent plane of $\Sigma$ is   
\begin{equation}
\begin{split}
\Psi_s(s,\varphi)&=(- u ' \sin  u \cos \varphi ,- u ' \sin  u \sin \varphi , u ' \cos  u , v'),\\
\Psi_\varphi(s,\varphi)&=(-\sin \varphi\cos  u ,\cos \varphi \cos  u ,0,0).
\end{split}
\end{equation}
Thus $g_{11}= 1$, $g_{12}=0$  and $g_{22}=\cos^2 u$, in particular, $\cos u\not=0$.  
As a consequence, the unit normal vector $N$ is \eqref{n2}. The computation of the coefficients of the second fundamental form gives 
\begin{equation}
\begin{split}
b_{11}&=\theta'\\
b_{12}&=0\\
b_{22}&=-\sin \theta \sin  u \cos  u .
\end{split}
\end{equation}
Hence we deduce that the expression of $H$ is \eqref{h2}.  

 \end{proof}
 
  \section{The class of $V$-solitons}\label{sec3}
  Let the vector field 
  \begin{equation}\label{vv}
  V=\partial_t.
  \end{equation}
 The fact that $V$ is tangent to the fibers of the submersion $\s^2\times\r\to\s^2$ makes that $V$ has special properties. For example,   $V$-solitons of $\s^2\times\r$ can be viewed as weighted minimal surfaces in a space with density. So, let $e^tdA$ and $e^t dV$ the area and volume of $\s^2\times\r$ with a weight $e^t$, where $t$ is the last coordinate of the space. Considering the energy functional  $\Omega\mapsto E(\Omega)=\int_\Omega e^t dA$ defined for compact subdomains $\Omega\subset\Sigma$, a critical point of this functional, also called a weighted minimal surface,   is a surface characterized by the equation $H-\langle N,\nabla t\rangle=0$, where   $\nabla$ is the gradient in $\s^2\times\r$. Since $\nabla t=\partial_t=V$, we have proved that   a weighted minimal surface in $(\s^2\times\r,e^t\langle,\rangle)$ is a $V$-soliton. A property of   weighted minimal surfaces  is they satisfy a principle of tangency as a consequence of the Hopf maximum principle for elliptic equations of divergence type. In our context, the tangency principle asserts that if two $V$-solitons $\Sigma_1$ and $\Sigma_2$ touch at some interior point $p\in\Sigma_1\cap \Sigma_2$ and one surface is in one side of the other around $p$, then $\Sigma_1$ and $\Sigma_2$ coincide in a  neighborhood of $p$. The following result proves that slices are the only closed $V$-solitons. 
  
 \begin{theorem} Slices $\s^2\times\{t_0\}$, $t_0\in\r$, are the only  closed (compact without boundary) $V$-solitons in $\s^2\times\r$.
 \end{theorem}
 
 \begin{proof} Let  $\psi\colon\Sigma\to\s^2\times\r$ be a closed $V$-soliton. Define on $\Sigma$ the height function $h\colon\Sigma\to\r$ by $h(q)=\langle\psi(q),\partial_t\rangle$. It is known that for any surface of $\s^2\times\r$, the Laplacian of $h$ is $\Delta h=H\langle N,\partial_t\rangle$ \cite{ro}.

 Using that $\Sigma$ is a $V$-soliton, then $\Delta h=\langle N,\partial_t\rangle^2=\langle N,V\rangle^2$. Integrating in $\Sigma$, the divergence theorem yields $\int_\Sigma \langle N,V\rangle^2=0$. Thus $\langle N,V\rangle=0$ on $\Sigma$ and $H=0$. In particular, $\Delta h=0$. By the maximum principle, $h$ is a constant function, namely $h(q)=t_0$, for some $t_0\in\r$. This proves that $\Sigma\subset \s^2\times\{t_0\}$ and thus, both surfaces coincide. 
 \end{proof}

 We begin with the  study of  $V$-solitons invariant by  the group  $\mathcal{T}$.   We prove that any vertical $V$-soliton is trivial  in the sense that it is a minimal surface, even more, we prove that it is a cylinder of type $\s^1\times\r\subset\s^2\times\r$.  
  
\begin{theorem}\label{t1}  Suppose that $\Sigma$ is a vertical  surface. Then $\Sigma$ is a $V$-soliton if and only its generating curve is a geodesic of $\s^2$ and $\Sigma$ is a vertical surface on a geodesic  of $\s^2$. 
\end{theorem}

 \begin{proof} Let $\Sigma$ be a vertical surface.  Since the vertical lines are fibers of the submersion $\s^2\times\r\to\s^2$,   the mean curvature $H$ of $\Sigma$ is $H=\kappa$, where $\kappa$ is the curvature of $\alpha$. Moreover, the unit normal vector is horizontal, hence $\langle N,V\rangle=0$. This proves the result. 
 \end{proof}

 We now study $V$-solitons of rotational type. As we have indicated in the previous section,  we can assume that the rotation axis is the $z$-axis. Thus a rotational surface $\Sigma$  can be parametrized by \eqref{p2}. 

An immediate example of rotational $V$-soliton   is the cylinder $\mathcal{C}=(\s^1\times\{0\})\times\r$. This surface corresponds with the curve $(u(s),v(s))=(0,s)$, $s\in\r$, in \eqref{a2}. Thus $\alpha(s)=(1,0,0,s)$  is the vertical line through the point $(1,0,0)\in\s^2$.  The unit normal  $N$ is orthogonal to $V$. Since the generating curve is a geodesic of $\s^2$, the surface is minimal, proving that $\mathcal{C}$ is  a $V$-soliton. This surface is also a vertical $R$-soliton (Thm. \ref{t1}). We now characterize rotational $V$-solitons in terms of its generating curve $\alpha$.

 \begin{proposition}\label{t2}
Let $\Sigma$ be  a rotational surface in $\s^2\times\r$. If $\Sigma$ is parametrized by \eqref{p2}, then    $\Sigma$ is a $V$-soliton  if and only if  the generating curve $\alpha$ satisfies
 \begin{equation}\label{s11}
 \left\{\begin{split}
  u'&=\cos\theta\\
  v'&=\sin\theta\\
 \theta'&=\sin\theta\tan u+\cos\theta
 \end{split}\right.
 \end{equation}
\end{proposition}

\begin{proof}  This is an immediate consequence of Prop. \ref{pr2}. Indeed, from the expression of $N$ in \eqref{n2}, we have 
$$\langle N,V\rangle=\cos\theta.$$
Using \eqref{h2}, then Eq. \eqref{eq1} is $ \theta'=\sin\theta\tan u+\cos\theta $. 
 
 \end{proof}

We now study the solutions of \eqref{s11}, describing their main geometric properties.  Recall that $\cos u\not=0$ by regularity of the surface (Prop. \ref{pr2}). Since the last equation of \eqref{s11} does not depend on $ v$, we can study the solutions  $\alpha$ of \eqref{s11}  projecting in the $( u,\theta)$-plane, which in turn leads to the planar autonomous ordinary  system
   \begin{equation}\label{s12}
\left\{ \begin{split}
  u'&=\cos\theta\\
 \theta'&=\sin\theta\tan u+\cos\theta.
 \end{split}\right.
 \end{equation}
 
 The phase plane of \eqref{s12} is depicted in Fig. \ref{phase1}.   By regularity of the surface, $u(s)\in (-\pi/2,\pi/2)$. Thus the phase plane of \eqref{s12} is the set
$$
A=\{(u,\theta)\colon u\in(\frac{\pi}{2},\frac{\pi}{2}),\theta\in \r\}.
$$ 
 The trajectories of \eqref{s12} are the solutions $\gamma(s)=(u(s),\theta(s))$ of \eqref{s12} when regarded in $A$ and once initial conditions $(u_0,\theta_0)\in A$ have been fixed. These trajectories  foliate $A$ as a consequence of the existence and uniqueness of the Cauchy problem of \eqref{s12}.

  \begin{figure}[hbtp]  
\begin{center}
\includegraphics[width=.4\textwidth]{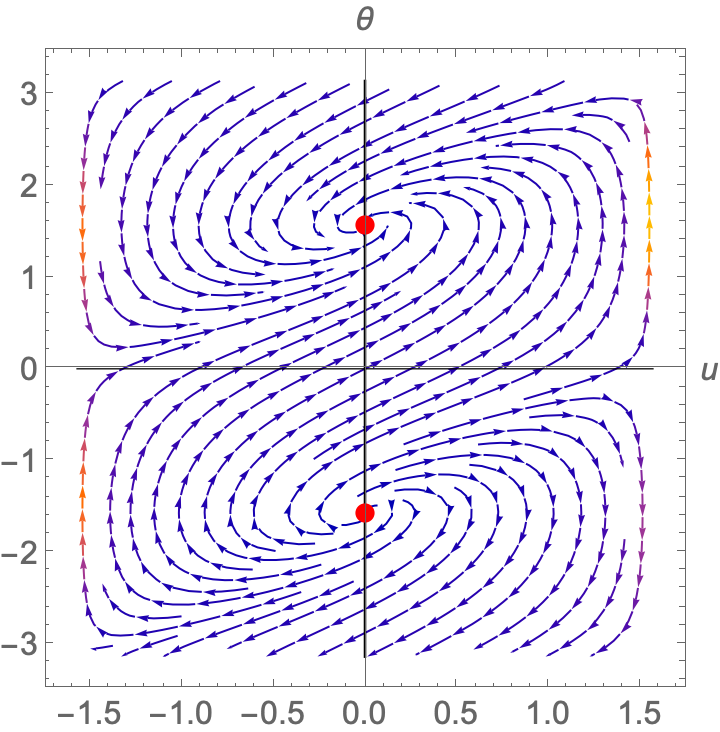} 
\end{center}
\caption{The $(u,\theta)$-phase plane of \eqref{s12}. The red points are the equilibrium points $(0,\pm\frac{\pi}{2})$, where the surface is the cylinder $\mathcal{C}$.  }
\label{phase1}
\end{figure} 

The equilibrium points of \eqref{s12} are $(u,\theta)=(0,\frac{\pi}{2})$ and $(u,\theta)=(0,-\frac{\pi}{2})$. The rest of equilibrium points can be obtained by translations by multiples of $\pi$ in the $\theta$-coordinate.  If  $(u,\theta)=(0,\frac{\pi}{2})$, then $ u(s)=0$, $ v(s)=s$. Thus   the generating curve $\alpha$ is the vertical fiber at $(1,0,0)\in\s^2$ parametrized with increasing  variable $s$, $v(s)=s$. For this curve, the corresponding surface is the vertical right cylinder $\mathcal{C}$ and this solution is already known. If $(u,\theta)=(0,-\frac{\pi}{2})$, then $v(s)=-s$ and the generating curve is again the above vertical line but parametrized by   decreasing variable $s$. The surface is   the cylinder $\mathcal{C}$ again.

 The qualitative behaviour of the trajectories near the equilibrium points are analyzed, as usually,   by the linearized system (see \cite{pe} as a general reference). At the point $(u,\theta)=(0,\frac{\pi}{2})$, we find
 $$\left(\begin{array}{ll}0&-1\\ 1& -1\end{array}\right)$$
 as the matrix of the linearized system. The eigenvalues of this matrix are the two  conjugate complex numbers $\frac{1}{2} (-1\pm i \sqrt{3})$. Since the real parts are negative, then  the point $(0,\frac{\pi}{2})$ is a stable spiral. Thus all the trajectories will move in towards the equilibrium point as $s$ increases. Similarly, for the point $(u,\theta)=(0,-\frac{\pi}{2})$, the matrix of the corresponding linearized system is  
 $$\left(\begin{array}{ll}0&1\\ -1& 1\end{array}\right).$$
Then the eigenvalues of this matrix are  $\frac{1}{2} (1\pm i \sqrt{3})$ and  the point $(0,-\frac{\pi}{2})$ is an unstable spiral.
 Since there are no more equilibrium points, every trajectory   starts in the unstable spiral  $(0,-\frac{\pi}{2})$ and ends in the stable spiral $(0,\frac{\pi}{2})$.

 In order to give initial conditions at $s=0$, notice that  if we do a vertical translation in $\r^4$ of the generating curve $\alpha$, the surface is a translated from the original. This vertical translation is simply adding a constant to the last coordinate function $v=v(s)$. Thus,  at the initial time $s=0$,  we can assume $v(0)=0$. On the other hand, the fact that the trajectories go from $(0,-\frac{\pi}{2})$ towards $(0,\frac{\pi}{2})$ implies that   the function $\theta$ attains the value $0$. As a first initial step, we can consider initial condition $\theta(0)=0$ when  the curve $\beta(s)=(u(s),v(s))$ starts at the $v$-axis. So, let 
\begin{equation}\label{ini}
u(0)=v(0)=\theta(0)=0.
\end{equation}
It is immediate   from \eqref{s11} that $(\bar{u}(s)=-u(-s), \bar{v}(s)=v(-s),\bar{\theta}(s)=-\theta(-s))$ is also a solution of \eqref{s11} with the same initial conditions \eqref{ini}. Thus   the graphic of $\beta$ is symmetric about the $v$-axis.

 Given initial conditions \eqref{ini}, we know that $(u(s),\theta(s))\to (0,\frac{\pi}{2})$. Then the right hand-sides of \eqref{s12} (also in \eqref{s11}) are bounded functions, proving that the domain of solutions is $\r$. Since $|v'(s)|\to 1$, then $\lim_{s\to\pm\infty}v(s)=\infty$ by symmetry of $\beta$. Thus $\lim_{s\to\pm\infty}\beta(s)=(0,\infty)$, that is $\beta$ is asymptotic to the $v$-axis at infinity. The projection of $\alpha$ in the factor $\s^2$ converges to $(1,0,0)$. This implies that $\Sigma$ is asymptotic to the cylinder $\mathcal{C}$.

 Because $(0,\pi/2)$ is a stable spiral, the function $\theta(s)$ converges to $\pi/2$ oscillating around this value, and the same occurs for the function $u(s)$ around $u=0$. In particular, the graphic of $\beta$ intersects  infinitely many times the $v$-axis. By the symmetry of $\beta$ with respect to the $v$-axis, we deduce that $\beta$ has (infinitely many) 
 self-intersections.

 We claim that the coordinate function $v(s)$ of $\beta$ has no critical points except $s=0$. We know 
 $$v''(s)=\theta'(s)\cos\theta(s)=\sin\theta(s)\cos\theta(s)\tan u(s)+\cos^2\theta(s).$$
 If  $v'(s)=0$ at $s=s_0$, then $\sin\theta(s_0)=0$, hence $v''(s_0)=1$. Thus all critical points are local minimum deducing that $s=0$ is the only minimum. Once we have prove that $v'\not=0$ for all $s\not=0$, then each branch of $\beta$, that is, $\beta(0,\infty)$ and $\beta(-\infty,0)$, are graphs on the $v$-axis. This proves that $\beta$ is a bi-graph on the $v$-axis. See Fig. \ref{fig1}, left.

It remains to study the case that  the curve $\beta$ does not start in the $v$-axis, that is, $u(0)=u_0$ with $u_0\not=0$.  Since the equilibrium points are of spiral type, the solutions of \eqref{s12} under this initial condition converge (or diverges) to the equilibrium points. Thus, the curve $\beta$ meets the $v$-axis being asymptotic to this axis.  

We summarize the above arguments. 

 \begin{figure}[hbtp] 
\begin{center}
\includegraphics[width=.15\textwidth]{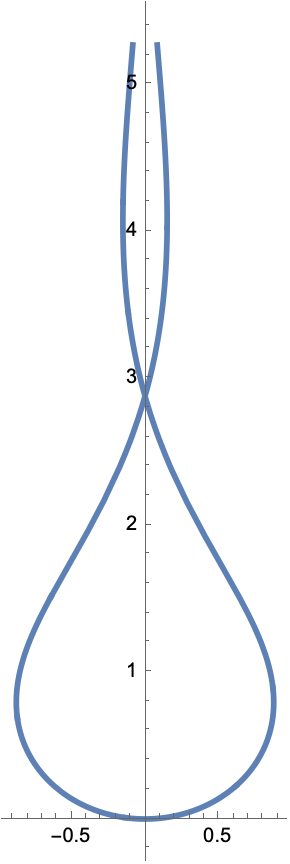}\qquad
\includegraphics[width=.15\textwidth]{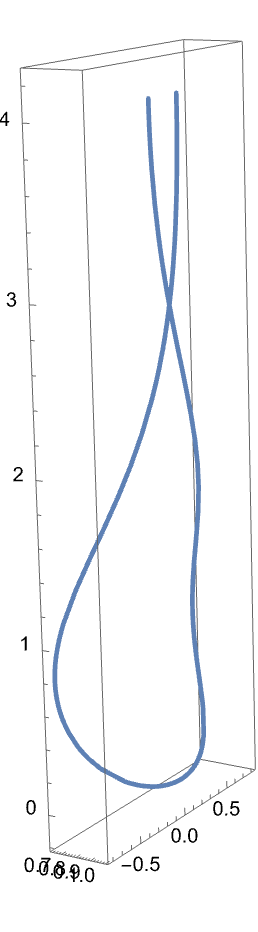}\qquad\includegraphics[width=.2\textwidth]{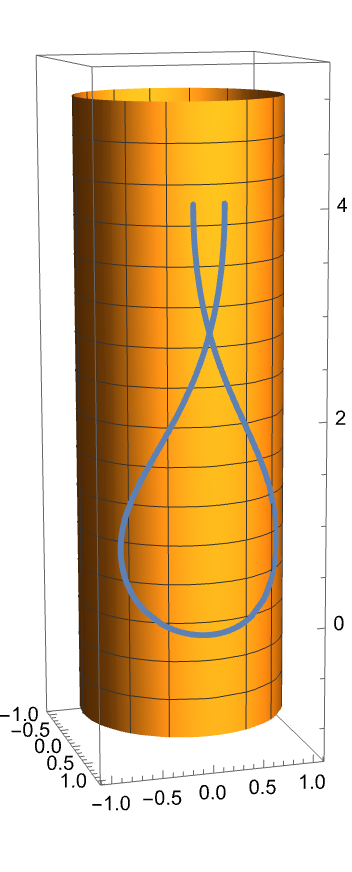}
\end{center}
\caption{Generating curves of rotational $V$-solitons. Left: the  curve $\beta$. 
Middle and right: projection of the generating curve $\alpha$ on the $xzt$-space (middle) and showing it as subset of   the cylinder $\s^1\times\r$ (right).  }
\label{fig1}
\end{figure} 

\begin{figure}[hbtp]
\begin{center}
\includegraphics[width=.3\textwidth]{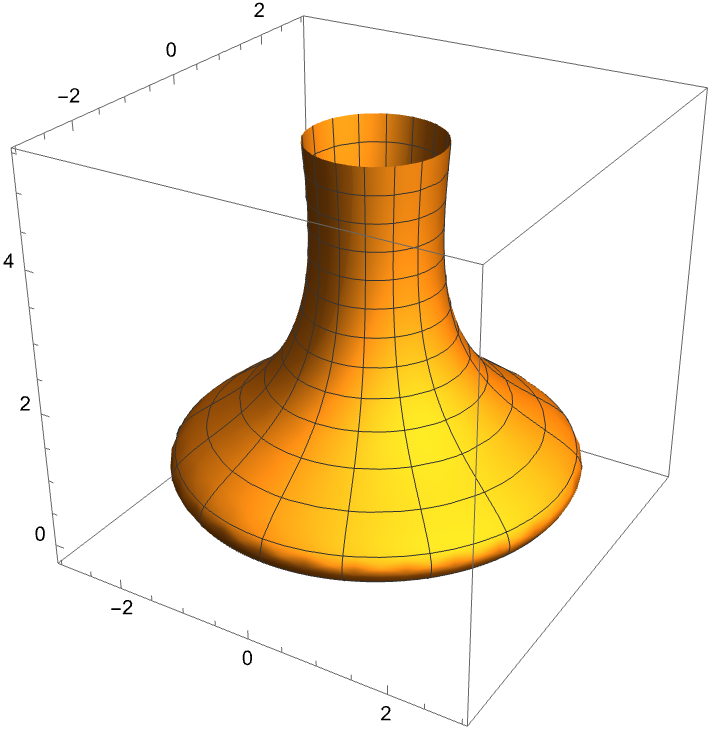}
\includegraphics[width=.3\textwidth]{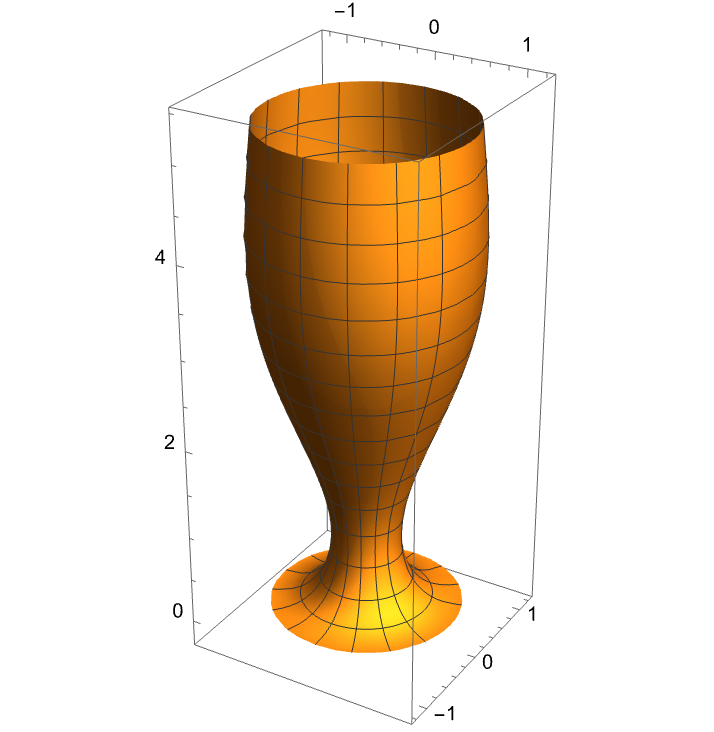}\includegraphics[width=.3\textwidth]{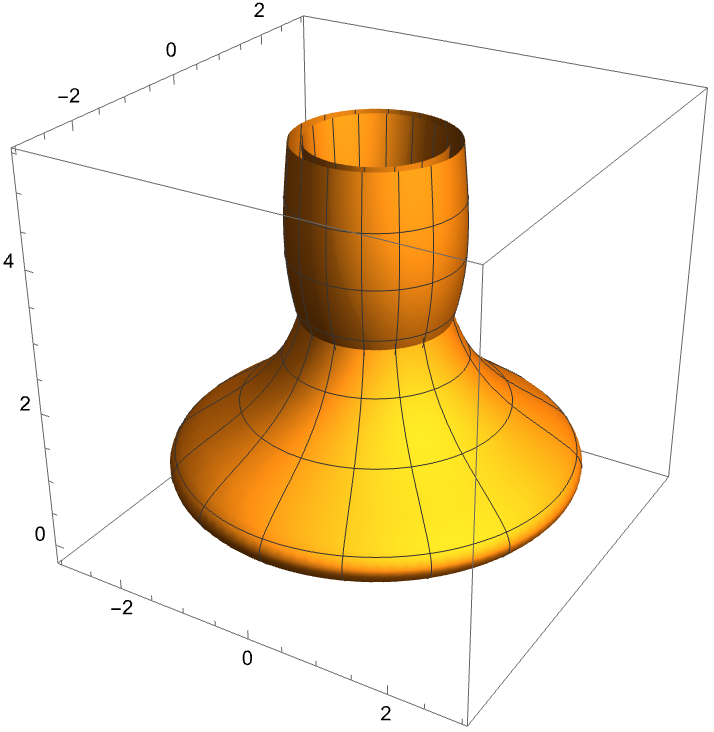}
\end{center}
\caption{ A rotational $V$-soliton after the stereographic projection $p_r$.  The surface after rotating $\beta$ in the interval $[0,\infty)$ (left) and   in the interval $(-\infty,0]$ (middle).   Right: the full surface.}
\label{fig2}
\end{figure} 

 \begin{theorem} \label{t23}
 Let $\Sigma$ be a rotational $V$-soliton. Then $\Sigma$ is the cylinder  $\mathcal{C}$ or   $\Sigma$ is parametrized by \eqref{p2} with the following properties:
 \begin{enumerate}
 \item The curve $\beta(s)=(u(s),v(s))$ has self-intersections and it is asymptotic to the $v$-axis at infinity. In case that   $\beta$ satisfies the initial conditions \eqref{ini}, then   $\beta$ is a symmetric bi-graph on the $v$-axis. 
 \item The surface $\Sigma$ is not embedded with infinitely many intersection points with the $z$-axis.
 \item The surface $\Sigma$ is asymptotic to the cylinder $\mathcal{C}$.
 \end{enumerate}

 \end{theorem}
 
  In Fig. \ref{fig2} we show the surface $\Sigma$ after a stereographic projection $p_r$ of the first factor $\s^2$ into $\r^2$, 
 $p_r\colon\s^2\times\r\to \r^2\times\r$, $p_r(x,y,z,t)=(\frac{x}{1-z},\frac{y}{1-z},t)$.

\section{The class of $R$-solitons}\label{sec4}

In this section we study $R$-solitons, where the  vector field $R$ is 
\begin{equation}\label{vr}
R=-y\partial_x+x\partial_y.\end{equation}
Following our scheme, we will classify $R$-solitons that are vertical surfaces and next, rotational surfaces. First, suppose that $\Sigma$ is a vertical $R$-soliton. We know that $\Sigma$ is  parametrized by \eqref{p1} and that the generating curve $\alpha$ is contained in $\s^2$, see  \eqref{a1}. A first example of    vertical $R$-soliton is a  vertical cylinder over the geodesic $\s^1\times\{0\}$ of $\s^2$. To be precise, let $(u(s),v(s))=(0,s)$. Then 
$\alpha(s)=(\cos s,\sin s,0)$ in \eqref{a1} and the surface is the vertical cylinder over $\alpha$ which we have denoted by $\mathcal{C}$ in the previous section. This surface is minimal and it is immediate that $N$ is orthogonal to $R$. Thus $\mathcal{C}$ is a $R$-soliton.  Recall that $\mathcal{C}$ is also a rotational $V$-soliton.

\begin{proposition}  Suppose that $\Sigma$ is a vertical   surface parametrized by \eqref{p1}.  Then $\Sigma$ is a $R$-soliton if and only if the generating curve $\alpha$ satisfies
 \begin{equation}\label{s21}
 \left\{\begin{split}
  u'&=\cos u\cos\theta\\
  v'&= \sin\theta \\
 \theta'&=\sin\theta\sin u+(\cos u)^2\cos\theta.
 \end{split}\right.
 \end{equation}  
\end{proposition}

\begin{proof}
The expression of the unit normal $N$ is given in \eqref{p1}. Thus
$$\langle N,R\rangle=-\cos u\cos\theta.$$
Since the expression of $H$ is given in \eqref{h1}, then Eq.  \eqref{eq1} is  $\theta'=\sin\theta\sin u+(\cos u)^2\cos\theta$, proving the result.
 \end{proof}
 
 As in the previous section, we project the solutions of \eqref{s21} on the $( u,\theta)$-plane, obtaining the autonomous system
  \begin{equation}\label{s22}
 \left\{\begin{split}
  u'&=\cos u\cos\theta\\
 \theta'&=\sin\theta\sin u+(\cos u)^2\cos\theta.
 \end{split}\right.
 \end{equation} 
 Equilibrium points are   $( u,\theta)=(0,\pm \pi/2)$ again, together the points $(u,\theta)=(\pm\frac{\pi}{2},0)$ and translations  of length $\pi$ of these points  in the $\theta$ variable. For the points $( u,\theta)=(0,\pm \pi/2)$, we have  $ u(s)=0$ and $v(s)=\pm s$ are solutions. In this case, we know that  $\Sigma$ is the vertical cylinder $\mathcal{C}$.  The   points $(u,\theta)=(\pm\frac{\pi}{2},0)$ do not correspond with surfaces because regularity is lost. In fact, coming back to the parametrization \eqref{p1}, the map $\Psi$ is the parametrization of the vertical fiber at $(0,0,1)\in\s^2$.

\begin{figure}[hbtp] 
\begin{center}
\includegraphics[width=.4\textwidth]{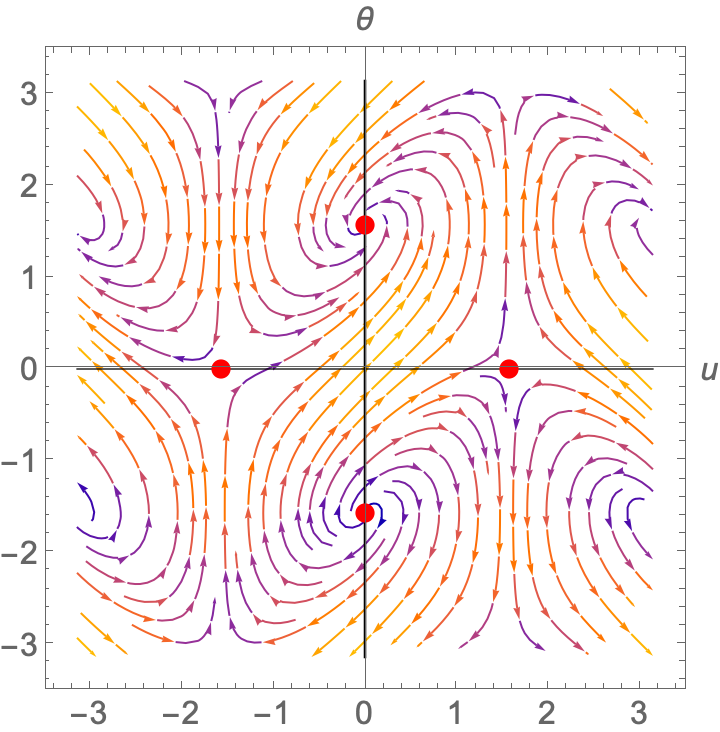} 
\end{center}
\caption{The $(u,\theta)$-phase plane of  \eqref{s22}.  The red points are the equilibrium points $(0,\pm\frac{\pi}{2})$. }\label{phase2}
\end{figure} 

The phase plane is the set $A=(-\pi,\pi)\times (-\pi,\pi)$ in the $(u,\theta)$-plane by the periodicity of $\theta$.  If we now compute   the linearized system at the points $(u,\theta)=(0,\pm\frac{\pi}{2})$, we find that they have the same character that  the ones of the system \eqref{s11}.   Thus we have that  $(u,\theta)=(0,\frac{\pi}{2})$ is a stable spiral and $(u,\theta)=(0,-\frac{\pi}{2})$  is an unstable spiral.  

For the points $( \frac{\pi}{2},0)$ and $( -\frac{\pi}{2},0)$, the linearized system is 
$$\left(\begin{array}{ll}-1&0\\ 0&1\end{array}\right),\quad \left(\begin{array}{ll}1&0\\ 0&-1\end{array}\right),$$
respectively. Since the eigenvalues are two real numbers with opposite signs, then both equilibrium points are saddle points. See Fig. \ref{phase2}. However, this does not affect to the reasoning since the arguments now are similar as in the proof of Thm. \ref{t23}. We omit the details. Figure \ref{fig3} shows generating curves of vertical $R$-solitons.

 \begin{theorem} \label{t32}
 Let $\Sigma$ be a vertical $R$-soliton. Then $\Sigma$ is the cylinder $\mathcal{C}$ or $\Sigma$ is parametrized by \eqref{p1} with the following properties:
 \begin{enumerate}
 \item  The curve $\beta(s)=(u(s),v(s))$ has self-intersections and it is asymptotic to the $v$-axis at infinity. In case that   $\beta$ satisfies the initial conditions \eqref{ini}, then   $\beta$ is a symmetric bi-graph on the $v$-axis. 
 \item The surface $\Sigma$ is not embedded with infinitely many intersection points with the $z$-axis.
 \item The surface $\Sigma$ is asymptotic to the cylinder $\mathcal{C}$.
 \end{enumerate}

 \end{theorem}

\begin{figure}[hbtp]
\begin{center}
\includegraphics[width=.12\textwidth]{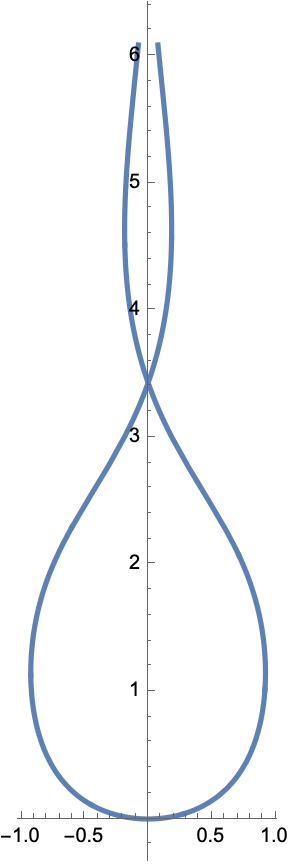}\qquad
\includegraphics[width=.25\textwidth]{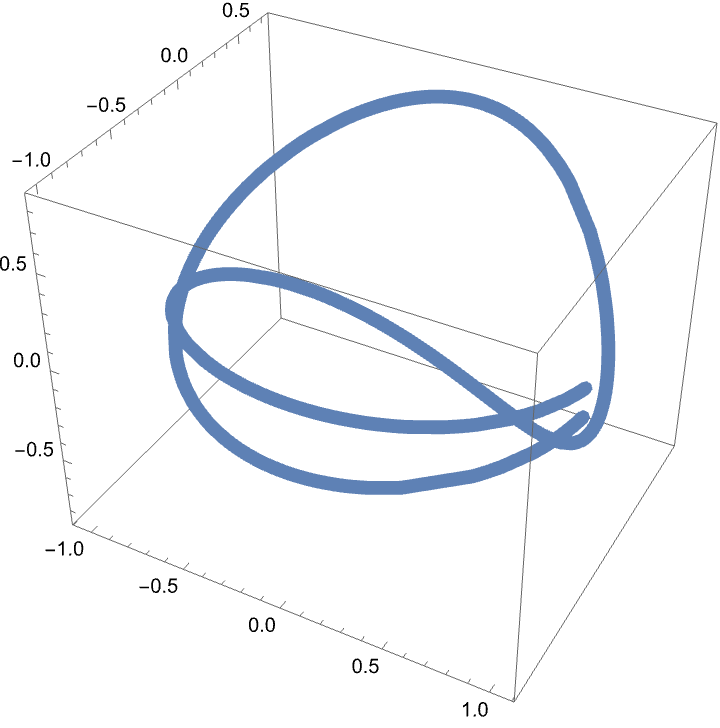}\qquad\includegraphics[width=.3\textwidth]{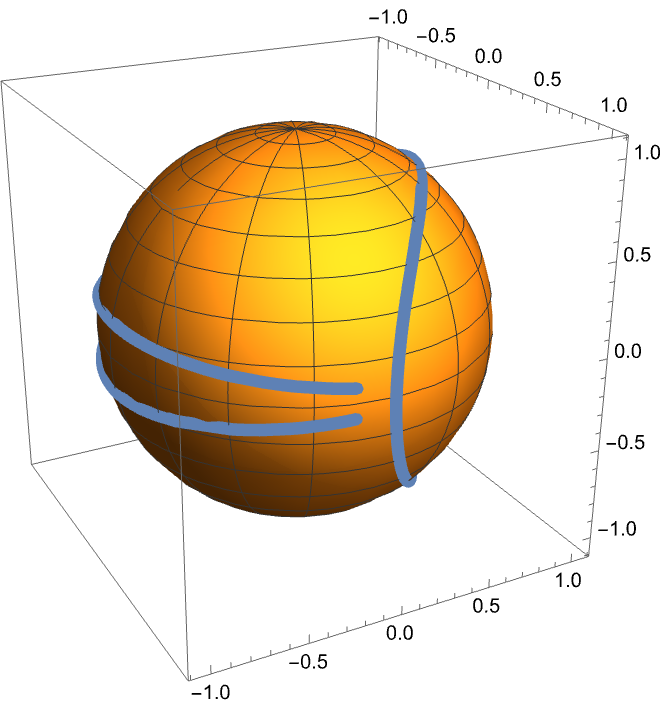}
\end{center}
\caption{Generating curves of vertical $R$-solitons. Left: solution curve $\beta(s)=(u(s),v(s))$. Middle: the generating curve $\alpha$.   Right: the generating curve $\alpha$ contained in the unit sphere $\s^2$. }
\label{fig3}
\end{figure} 

The second type of $R$-solitons are those surfaces that are invariant by a one-parameter group of rotations of the first factor.   Since we have defined in \eqref{vr} the vector field $R$ as the rotation about the direction $(0,0,1)\in\s^2$, we cannot {\it a priori} fix the rotational axis of the surface.

\begin{theorem} The only rotational $R$-solitons are:
\begin{enumerate}
\item Slices $\s^2\times\{t_0\}$, $t_0\in\r$, viewed as rotational surfaces with respect to any axis of $\s^2$ and;
\item Rotational minimal  surfaces about the $z$-axis.  
\end{enumerate}

\end{theorem}

\begin{proof} Let $\Sigma$ be a rotational $R$-soliton.  In order to have manageable computations of the mean curvature $H$ and the unit normal $N$ of $\Sigma$, we will assume in this proof that the rotation axis of $\Sigma$ is the $z$-axis. Thus we are assuming that  the surface is parametrized by \eqref{p2}. Thus the vector field $R$ is now arbitrary and with no {\it a priori} relation with the $z$-axis. The vector field $R$ is determined by   an orthonormal basis  $B=\{E_1,E_2,E_3\}$  of $\r^3$. With respect to $B$, the vector field $R$ can be expressed by 
$$R(x_1,x_2,x_3,t)=-x_2 E_1+x_1 E_2,$$
where $(x_1,x_2,x_3)$ are coordinates of $\s^2$ with respect to $B$. 

We now write $E_1$ and $E_2$ with respect to the canonical basis of $\r^3$, 
$$E_i=(\cos m_i\sin n_i,\cos m_i \sin n_i,\sin m_i),\quad i=1,2,$$
where $m_i,n_i\in\r$. The  unit normal $N$ and the mean curvature $H$  of $\Sigma$ were computed in \eqref{n2} and \eqref{h2}, respectively. Then 
\begin{equation}\label{nrr}
\begin{split}
\langle N,R\rangle&=-\langle \Psi,E_2\rangle \langle N,E_1\rangle+\langle \Psi,E_1\rangle\langle N,E_2\rangle\\
=& (\sin m_1 \cos m_2 \sin n_2-\cos m_1 \sin n_1 \sin m_2) \sin \theta\sin\varphi\\
&+  (\sin m_1\cos m_2 \cos n_2-\cos m_1 \cos n_1 \sin m_2)\sin\theta\cos\varphi.
\end{split}
\end{equation}
Looking now Eq. \eqref{eq1}, we have that in the expression of right hand-side of \eqref{eq1}, that is, $\langle N,R\rangle$, formula \eqref{nrr}, the variable $\varphi$ do appear. However in the left hand-side of \eqref{eq1}, the mean curvature $H$, formula \eqref{h2},    does not depend on $\varphi$. This implies that the coefficients of $\sin\varphi$ and $\cos\varphi$ in \eqref{nrr} must vanish. Both coefficients contain the factor $\sin\theta$. This gives the following discussion of cases.  
\begin{enumerate}
\item Case $\sin\theta(s)=0$ for all $s$. Then $u(s)=s$ and $v(s)$ is a constant function, $v(s)=t_0$, $t_0\in\r$. This proves that $\Sigma$ is a slice $\s^2\times\{t_0\}$.
\item Case $\sin\theta(s_0)\not=0$ at some $s_0$. Then in an interval around $s=s_0$, we deduce 
$$\sin m_1 \cos m_2 \sin n_2-\cos m_1 \sin n_1 \sin m_2=0,$$
$$\sin m_1\cos m_2 \cos n_2-\cos m_1 \cos n_1 \sin m_2=0.$$
Both identities imply   $E_1\times E_2=(0,0,1)$. Thus $R$ coincides with the vector field defined in \eqref{vr} and the rotation axis is the $z$-axis. Moreover, the right hand-side of \eqref{eq1} is $0$, proving that the surface is minimal.    This proves the result.
\end{enumerate}

\end{proof}
Minimal surfaces in $\s^2\times\r$ of rotational type with respect to an axis in the first factor were classified by Pedrosa and Ritor\'e \cite{pr}. 
\section*{Acknowledgements}  

Rafael L\'opez  is a member of the IMAG and of the Research Group ``Problemas variacionales en geometr\'{\i}a'',  Junta de Andaluc\'{\i}a (FQM 325). This research has been partially supported by MINECO/MICINN/FEDER grant no. PID2020-117868GB-I00, and by the ``Mar\'{\i}a de Maeztu'' Excellence Unit IMAG, reference CEX2020-001105- M, funded by MCINN/AEI/10.13039/501100011033/ CEX2020-001105-M.
Marian Ioan Munteanu is thankful to Romanian Ministry of Research, Innovation and Digitization, within Program 1 – Development of the national RD system, Subprogram 1.2 – Institutional Performance – RDI excellence funding projects, Contract no.11PFE/30.12.2021, for financial support.

\end{document}